\documentclass[twoside,leqno, british,11pt]
{amsart}
\usepackage{amsmath, amssymb, amsbsy, amsthm}
\usepackage{babel,eucal,url,amssymb,
enumerate,amscd,
}

\usepackage[T1]{fontenc}

\DeclareMathOperator{\Id}{Id} 
\DeclareMathOperator{\Log}{Log}

\begin{document}

\title[Curvature tensors of K\"ahler type]
{Almost contact B-metric manifolds with curvature
tensors of K\"ahler type}

\author[M. Manev, M. Ivanova]{Mancho Manev, Miroslava Ivanova}


\frenchspacing

\newcommand{\ie}{i.e. }
\newcommand{\X}{\mathfrak{X}}
\newcommand{\W}{\mathcal{W}}
\newcommand{\F}{\mathcal{F}}
\newcommand{\T}{\mathcal{T}}
\newcommand{\LL}{\mathcal{L}}
\newcommand{\TT}{\mathfrak{T}}
\newcommand{\M}{(M,\f,\xi,\eta,g)}
\newcommand{\Lf}{(G,\f,\xi,\eta,g)}
\newcommand{\R}{\mathbb{R}}
\newcommand{\s}{\mathfrak{S}}
\newcommand{\n}{\nabla}
\newcommand{\nn}{\tilde{\nabla}}
\newcommand{\tg}{\tilde{g}}
\newcommand{\f}{\varphi}
\newcommand{\D}{{\rm d}}
\newcommand{\id}{{\rm id}}
\newcommand{\al}{\alpha}
\newcommand{\bt}{\beta}
\newcommand{\gm}{\gamma}
\newcommand{\dt}{\delta}
\newcommand{\lm}{\lambda}
\newcommand{\ta}{\theta}
\newcommand{\om}{\omega}
\newcommand{\ea}{\varepsilon_\alpha}
\newcommand{\eb}{\varepsilon_\beta}
\newcommand{\eg}{\varepsilon_\gamma}
\newcommand{\sx}{\mathop{\mathfrak{S}}\limits_{x,y,z}}
\newcommand{\norm}[1]{\left\Vert#1\right\Vert ^2}
\newcommand{\nf}{\norm{\n \f}}
\newcommand{\Span}{\mathrm{span}}
\newcommand{\grad}{\mathrm{grad}}
\newcommand{\thmref}[1]{The\-o\-rem~\ref{#1}}
\newcommand{\propref}[1]{Pro\-po\-si\-ti\-on~\ref{#1}}
\newcommand{\secref}[1]{\S\ref{#1}}
\newcommand{\lemref}[1]{Lem\-ma~\ref{#1}}
\newcommand{\dfnref}[1]{De\-fi\-ni\-ti\-on~\ref{#1}}
\newcommand{\corref}[1]{Corollary~\ref{#1}}



\numberwithin{equation}{section}
\newtheorem{thm}{Theorem}[section]
\newtheorem{lem}[thm]{Lemma}
\newtheorem{prop}[thm]{Proposition}
\newtheorem{cor}[thm]{Corollary}

\theoremstyle{definition}
\newtheorem{defn}{Definition}[section]

\hyphenation{Her-mi-ti-an ma-ni-fold ah-ler-ian}




\begin{abstract}
On 5-dimensional almost contact B-metric manifolds, the form of
any $\f$-K\"ahler-type tensor (i.e. a tensor satisfying the
properties of the curvature tensor of the Levi-Civita connection
in the special class of the parallel structures on the manifold)
is determined. The associated 1-forms are derived by the scalar
curvatures of the $\f$-K\"ahler-type tensor for the $\f$-canonical
connection on the manifolds from the main classes with closed
1-forms.
\end{abstract}



\maketitle

\begin{center}
\end{center}

\section*{Introduction}

The curvature properties of the almost contact B-metric manifolds
are investigated with respect to the Levi-Civita connection $\n$
and other linear connection preserving the structures of the
manifold. A significant role play such connections, which
curvature tensors possess the properties of the curvature tensor
of $\n$ in the class with $\n$-parallel structures.

The present paper\footnote{Partially supported by project
NI11-FMI-004 of the Scientific Research Fund, Paisii Hilendarski
University of Plovdiv, Bulgaria} is organized as follows. In
Sec.~1, we give some necessary facts about the considered
manifolds.
Sec.~2 is devoted to the $\f$-K\"ahler-type tensors, i.e. the
tensors satisfying the properties of the curvature tensor of $\n$
in the special class $\F_0$.
In Sec.~3, it is determined the form of any $\f$-K\"ahler-type
tensor $L$ on a 5-dimensional manifold under consideration.
In Sec.~4, it is proved that the associated 1-forms $\ta$ and
$\ta^*$ are derived by the non-$\f$-holomorphic pair of scalar
curvatures of the $\f$-K\"ahler-type tensor for the $\f$-canonical
connection on the manifolds from the main classes with closed
1-forms.
In Sec.~5, some of the obtained results are illustrated by a known
example.

\section{Preliminaries}

Let $(M,\f,\xi,\eta,g)$ be an almost contact manifold with
B-metric or an \emph{almost contact B-metric manifold}, \ie $M$ is
a $(2n+1)$-dimensional differen\-tia\-ble manifold with an almost
contact structure $(\f,\xi,\eta)$ consisting of an endomorphism
$\f$ of the tangent bundle, a vector field $\xi$, its dual 1-form
$\eta$ as well as $M$ is equipped with a pseudo-Riemannian metric
$g$  of signature $(n,n+1)$, such that the following relations are
satisfied
\begin{gather}
\f\xi = 0,\quad \f^2 = -\Id + \eta \otimes \xi,\quad
\eta\circ\f=0,\quad \eta(\xi)=1, \label{str}\nonumber\\%
g(x, y ) = - g(\f x, \f y ) + \eta(x)\eta(y) \label{g}\nonumber
\end{gather}
for arbitrary $x$, $y$ of the algebra $\X(M)$ on the smooth vector
fields on $M$.

Further, $x$, $y$, $z$ will stand for arbitrary elements of
$\X(M)$.

The associated metric $\tilde{g}$ of $g$ on $M$ is defined by
\begin{equation*}\label{tg}
\tilde{g}(x,y)=g(x,\f y)\allowbreak+\eta(x)\eta(y).
\end{equation*}
Both metrics $g$ and $\tilde{g}$ are necessarily of signature
$(n,n+1)$. The manifold $(M,\f,\xi,\eta,\tilde{g})$ is also an
almost contact B-metric manifold.

The structural tensor $F$ of type (0,3) on $\M$ is defined by
$
F(x,y,z)=g\bigl( \left( \nabla_x
\f \right)y,z\bigr). $ 
It has the following properties:
\begin{equation*}\label{F-prop}
\begin{split}
F(x,y,z)&=F(x,z,y)
=F(x,\f y,\f z)+\eta(y)F(x,\xi,z) +\eta(z)F(x,y,\xi).
\end{split}
\end{equation*}

The following 1-forms are associated with $F$:
\begin{equation*}\label{titi}
\begin{array}{c}
\ta(z)=g^{ij}F(e_i,e_j,z),\quad \ta^*(z)=g^{ij}F(e_i,\f e_j,z),\quad
\om(z)=F(\xi,\xi,z),
\end{array}
\end{equation*}
where $g^{ij}$ are the components of the inverse matrix of $g$
with respect to a basis $\left\{e_i;\xi\right\}$
$(i=1,2,\dots,2n)$ of the tangent space $T_pM$ of $M$ at an
arbitrary point $p\in M$. Obviously, the equality $\om(\xi)=0$ and
the relation $\ta^*\circ\f=-\ta\circ\f^2$ are always valid.

A classification of the almost contact manifolds with B-metric
with respect to $F$ is given in \cite{GaMiGr}. This classification
includes eleven basic classes $\F_1$, $\F_2$, $\dots$, $\F_{11}$.
Their intersection is the special class $\F_0$ determined by
$F(x,y,z)=0$. Hence $\F_0$ is the class of almost contact B-metric
manifolds with $\n$-parallel structures, i.e.
$\n\f=\n\xi=\n\eta=\n g=\n \tilde{g}=0$.

In the present paper we consider the manifolds from the so-called
main classes $\F_1$, $\F_4$, $\F_5$ and $\F_{11}$, shortly the
$\F_i$-manifolds $(i=1,4,5,11)$. These classes are the only
classes where the tensor $F$ is expressed by the metrics $g$ and
$\tilde{g}$. They are defined as follows:
\begin{equation*}\label{Fi}
\begin{split}
\F_{1}:\quad &F(x,y,z)=\frac{1}{2n}\bigl\{g(x,\f y)\ta(\f z)+g(\f
x,\f y)\ta(\f^2 z)
\bigr\}_{(y\leftrightarrow z)};\\[4pt]
\F_{4}:\quad &F(x,y,z)=-\frac{\ta(\xi)}{2n}\bigl\{g(\f x,\f y)\eta(z)+g(\f x,\f z)\eta(y)\bigr\};\\[4pt]
\F_{5}:\quad &F(x,y,z)=-\frac{\ta^*(\xi)}{2n}\bigl\{g( x,\f y)\eta(z)+g(x,\f z)\eta(y)\bigr\};\\[4pt]
\F_{11}:\quad
&F(x,y,z)=\eta(x)\left\{\eta(y)\om(z)+\eta(z)\om(y)\right\},
\end{split}
\end{equation*}
where (for the sake of brevity) we use the denotation
$\{A(x,y,z)\}_{(y\leftrightarrow z)}$ --- instead of
$\{A(x,y,z)+A(x,z,y)\}$ for any tensor $A(x,y,z)$.

Let us remark that the class
$\F_1\oplus\F_4\oplus\F_5\oplus\F_{11}$ is the odd-dimensional
analogue of the class $\W_1$ of the conformal K\"ahler manifolds
of the almost complex manifold with Norden metric, introduced in
\cite{GrMeDj}.

\section{Curvature-like tensors}

Let $R=\left[\n,\n\right]-\n_{[\ ,\ ]}$ be the curvature
(1,3)-tensor of the Levi-Civita connection $\nabla$. We denote the
curvature $(0,4)$-tensor by the same letter: $R(x,y,z,w)$
$=g(R(x,y)z,w)$.

The Ricci tensor $\rho$ and the scalar curvature $\tau$ for $R$ as
well as their associated quantities are defined respectively by
\begin{equation*}\label{rho}
\begin{array}{ll}
    \rho(y,z)=g^{ij}R(e_i,y,z,e_j),\qquad
    &\tau=g^{ij}\rho(e_i,e_j),\\[4pt]
    \rho^*(y,z)=g^{ij}R(e_i,y,z,\f e_j),\qquad
    &\tau^*=g^{ij}\rho^*(e_i,e_j).\\[4pt]
\end{array}
\end{equation*}

\begin{defn}[\cite{ManGri2}]
Each (0,4)-tensor $L$ on $\M$ having the following properties is
called a \emph{curvatu\-re-like tensor}:
\begin{gather}
    L(x,y,z,w)=-L(y,x,z,w)=-L(x,y,w,z),\label{R1}
    \\%
    \sx L(x,y,z,w)=0. 
    \label{R2}
\end{gather}
\end{defn}

The above properties are a characteristic of the curvature tensor
$R$.

Similarly of \eqref{rho}, the Ricci tensor, the scalar curvature
and their associated quantities are determined for each
curvatu\-re-like tensor $L$.

\begin{defn}[\cite{ManGri2}]
A curvature-like tensor $L$ on $\M$ is called a
\emph{$\f$-K\"ahler-type tensor} if it satisfies the condition
\begin{equation}\label{LK}
L(x,y,\f z,\f w)=-L(x,y,z,w).
\end{equation}
\end{defn}

This property is a characteristic of $R$ on a $\F_0$-manifold.
Moreover, \eqref{LK} is similar to the  property for a
K\"ahler-type tensor with respect to $J$ on an almost complex
manifold with Norden metric (\cite{GaGrMi}).

\begin{lem}\label{lem-Kaehler}
If $L$ is a $\f$-K\"ahler-type tensor on $\M$, then the following
properties are valid:
\begin{gather}
L(\f x,\f y,z,w)=L(x,\f y,\f z,w)=-L(x,y,z,w),\label{L1}\\[4pt]
L(\xi,y,z,w)=L(x,\xi,z,w)=L(x,y,\xi,w)=L(x,y,z,\xi)=0,\label{L2}\\[4pt]
L(\f x,y,z,w)=L(x,\f y,z,w)=L(x,y,\f z,w)=L(x,y,z,\f w).\label{L3}
\end{gather}
\end{lem}
\begin{proof}
The equalities \eqref{L1} and \eqref{L2} follow immediately from
\eqref{R1}, \eqref{R2} and \eqref{LK}. The properties \eqref{L1}
and \eqref{L2} imply \eqref{L3}.
\end{proof}

We consider an associated tensor  $L^*$ of $L$ by the equality
\[
L^*(x,y,z,w)=L(x,y,z,\f w).
\]
Let us remark, the tensor $L^*$ is not a curvature-like tensor at
all. If $L$ is a $\f$-K\"ahler-type tensor, then $L^*$ is also a
$\f$-K\"ahler-type tensor. Then the properties in
\lemref{lem-Kaehler} are valid for $L^*$. Obviously, the
associated tensor of $L^*$, i.e. $\left(L^*\right)^*$, is $-L$.
Consequently, we have the following
\begin{cor}
Let $L$ and its associated tensor $L^*$ be $\f$-K\"ahler-type
tensors on $\M$. Then we have
\[
\begin{array}{ll}
\rho(L^*)=\rho^*(L),\qquad & \rho^*(L^*)=-\rho(L),\\[4pt]
\tau(L^*)=\tau^*(L),\qquad & \tau^*(L^*)=-\tau(L).
\end{array}
\]
\end{cor}

\subsection{Examples of curvature-like tensors of $\f$-K\"ahler
type}

Let us consider the following basic tensors of type (0,4) derived
by the structural tensors of $\M$ and an arbitrary tensor $S$ of
type (0,2):
\[
\begin{split}
\psi_1(S)(x,y,z,w)&=\bigl\{g(y,z)S(x,w)+g(x,w)S(y,z)\bigr\}_{[x\leftrightarrow
y]},\\[4pt]
\psi_2(S)(x,y,z,w)&=\bigl\{g(y,\f z)S(x, \f w)+g(x, \f w)S(y,\f
z)\bigr\}_{[x\leftrightarrow y]},
\\[4pt]
\psi_3(S)(x,y,z,w)&=-\bigl\{g(y,z)S(x,\f w)+g(y, \f z)S( x,w)
\\[4pt]
&\phantom{=-\bigl\{}+g(x,\f w)S(y,z)+g( x,w)S(y, \f
z)\bigr\}_{[x\leftrightarrow y]},
\\[4pt]
\psi_4(S)(x,y,z,w)&=\bigl\{\eta(y)\eta(z)S(x,w)+\eta(x)\eta(w)S(y,
z)\bigr\}_{[x\leftrightarrow y]},\\[4pt]
\psi_5(S)(x,y,z,w)&=\bigl\{\eta(y)\eta(z)S(x,\f w)
+\eta(x)\eta(w)S(y,\f z)\bigr\}_{[x\leftrightarrow y]},
\end{split}
\]
where we use the denotation $\left\{A(x, y,
z)\right\}_{[x\leftrightarrow y]}$ instead of the difference $A(x,
y, z) - A(y,x,z)$ for any tensor $A(x, y, z)$. The tensor
$\psi_1(S)$ coincides with the known Kulkarni-Nomizu product of
the tensors $g$ and $S$.

The five tensors $\psi_i(S)$ are not curvature-like tensors at
all. In \cite{ManGri2} and \cite{Man-diss}, it is proved that on
an almost contact B-metric manifold:
\begin{enumerate}
    \item $\psi_1(S)$ and $\psi_4(S)$ are curvature-like tensors
    if and only if $S(x,y)=S(y,x)$;
    \item $\psi_2(S)$ and $\psi_5(S)$ are curvature-like tensors
    if and only if $S(x,\f y)\allowbreak{}=S(y,\f x)$;
    \item $\psi_3(S)$ is a curvature-like tensor
    if and only if $S(x,y)=S(y,x)$ and $S(x,\f y)=S(y,\f x)$.
\end{enumerate}
Moreover, both of the tensors $\psi_1(S)-\psi_2(S)-\psi_4(S)$ and
$\psi_3(S)+\psi_5(S)$ are of $\f$-K\"ahler type if and only if the
tensor $S$ is symmetric and hybrid with respect $\f$, i.e.
$S(x,y)=S(y,x)$ and $S(x,y)=-S(\f x,\f y)$. In this case, their
associated tensors are the following:
\[
\begin{array}{l}
\left(\psi_1-\psi_2-\psi_4\right)^*(S)=-\left(\psi_3+\psi_5\right)(S),\\[4pt]
\left(\psi_3+\psi_5\right)^*(S)=\left(\psi_1-\psi_2-\psi_4\right)(S).
\end{array}
\]

The following tensors $\pi_i$ ($i=1, 2, \dots, 5$), derived only
by the metric tensors of $\M$, play an important role in
differential geometry of an almost contact B-metric manifold:
\[
\pi_i=\frac{1}{2}\psi_i(g),\; (i=1,2,3); \qquad \pi_i=\psi_i(g),\;
(i=4,5).
\]
In \cite{ManGri2}, it is proved that $\pi_i$ ($i=1, 2, \dots, 5$)
are curvature-like tensors and the tensors
\[
L_1=\pi_1-\pi_2-\pi_4,\qquad L_2=\pi_3+\pi_5
\]
are $\f$-K\"ahler-type tensors. Their associated
$\f$-K\"ahler-type tensors are as follows
\[
L_1^*=-L_2,\qquad L_2^*=L_1.
\]

\section{$\f$-K\"ahler-type tensors on a 5-dimensional almost contact B-metric manifold}

Let $\al$ be a non-degenerate totally real section in $T_pM$, $p
\in M$, and $\al$ be orthogonal to $\xi$ with respect to $g$, i.e.
$\al\perp\f\al$, $\al\perp\xi$. Let $k(\al; p)(L)$ and $k^*(\al;
p)(L)$ be the scalar curvatures of $\al$ with respect to a
curvature-like tensor $L$, i.e.
\[
k(\al; p)(L)=\frac{L(x,y,y,x)}{\pi_1(x,y,y,x)}, \qquad k^*(\al;
p)(L)=\frac{L(x,y,y,\f x)}{\pi_1(x,y,y,x)},
\]
where $\{x,y\}$ is an arbitrary basis of $\al$.

We recall two known propositions for constant sectional
curvatures.
\begin{thm}[\cite{NakGri-97}]\label{thm:A}
Let $\M$ $(\dim M \geq 5)$ be an almost contact B-metric
$\F_0$-manifold. Then $M$ is of constant totally real sectional
curvatures $\nu=\nu(p)(R) = k(\al; p)(R)$ and $\nu^*=\nu^*(p)(R) =
˜k^*(\al; p)(R)$ if and only if $R = \nu L_1 + ˜\nu^* L_2$. Both
functions $\nu$ and $\nu^*$ are constant if $M$ is connected and
$\dim M \geq 7$.
\end{thm}

\begin{thm}[\cite{NakMan14}]\label{thm:B}
Each 5-dimensional almost contact B-metric $\F_0$-manifold
 has point-wise constant sectional curvatures
\[
 \nu(p)(R) =
k(\al; p)(R),\qquad \nu^*(p)(R) = ˜k^*(\al; p)(R).
\]
\end{thm}

In this relation, we give the following
\begin{thm}\label{thm:L5}
Let $\M$ be a 5-dimensional almost contact B-metric manifold. Then
each $\f$-K\"ahler-type tensor has the form $$L=\nu L_1+\nu^*
L_2,$$ where $\nu=\nu(L)$ and $\nu^*=\nu^*(L)=\nu(L^*)$ are the
sectional curvatures of the totally real 2-planes orthogonal to
$\xi$ in $T_pM$, $p\in M$, with respect to $L$. Moreover, $\M$ is
of point-wise contact sectional curvatures of the totally real
2-planes orthogonal to $\xi$ with respect to $L$.
\end{thm}
\begin{proof}
Let $\{e_1, e_2, \f e_1, \f e_2,\xi\}$ be an adapted $\f$-basis of
$T_pM$ with respect to $g$, i.e.
\[
\begin{array}{c}
-g(e_1,e_1)=-g(e_2,e_2)=g(\f e_1,\f e_1)=g(\f e_2,\f
e_2)=1,\\[4pt]
g(e_i,\f e_j)=0, \quad \eta(e_i)=0\quad (i,j\in\{1,2\}).
\end{array}
\]
Then an arbitrary vector in $T_pM$ has the form $x = x^1e_1 +
x^2e_2 + \tilde{x}^1\f e_1 +\tilde{x}^2\f e_2 + \eta(x)\xi$. Using
 properties \eqref{R1}, \eqref{R2} and \eqref{LK} for $L(x, y,
z,w)$, we obtain immediately $L = \nu L_1 + ˜\nu^* L_2$, where
$\nu = L(e_1, e_2, e_2, e_1)$, $\nu^* = L(e_1, e_2, e_2, \f
e_1)=\nu(L^*) = L^*(e_1, e_2, e_2, e_1)$ are the sectional
curvatures of $\al$ with respect to $L$, because $\pi_1(e_1, e_2,
e_2, e_1)=1$.

Then, if $\{x,y\}$ is an adapted $\f$-basis of an arbitrary
totally real 2-plane $\al$ orthogonal to $\xi$, i.e.
\[
g(x,y)=g(x,\f x)=g(x,\f y)=g(y,\f y)=\eta(x)=\eta(y)=0,
\]
we get $k(\al;p)(L)=\nu(p)(L)$, $k^*(\al;p)(L)=\nu^*(p)(L)$,
taking into account the expression $L = \nu L_1 + ˜\nu^* L_2$.
Therefore, $\M$ is of point-wise contact sectional curvatures of
$\al$ with respect to $L$.
\end{proof}

The restriction of \thmref{thm:L5} in $\F_0$ confirms
\thmref{thm:A} because $R$ is a $\f$-K\"ahler-type tensor on a
$\F_0$-manifold.

\subsection{Curvature tensor of a natural connection on a 5-dimension\-al almost contact B-metric manifold}

In \cite{Man31}, it is introduced the notion of a \emph{natural
connection} on the manifold $(M,\f,\xi,\eta,g)$ as a linear
connection $D$, with respect to which the almost contact structure
$(\f,\xi,\eta)$ and the B-metric $g$ are parallel, \ie
$D\f=D\xi=D\eta=Dg=0$. According to \cite{ManIv36}, a necessary
and sufficient condition a linear connection $D$ to be natural on
$\M$ is $D\f=Dg=0$.

Let $K$ be curvature tensor of a natural connection $D$ with
torsion $T$. Then $K$ satisfies \eqref{R1} and \eqref{LK}. Instead
of \eqref{R2}, we have the following form of the first Bianchi
identity (\cite{KoNo})
\[
\sx K(x,y,z,w)=\sx
\left\{T(T(x,y),z,w)+\left(D_xT\right)(y,z,w)\right\}.
\]

If we set the condition $\sx K(x,y,z,w)=0$ as for the curvature
tensor $R$ then $K$ is a $\f$-K\"ahler-type tensor and satisfies
the condition of \thmref{thm:L5}. Therefore, we obtain

\begin{cor}\label{cor:K5}
Let $\M$ be a 5-dimensional almost contact B-metric manifold with
a natural connection $D$ with curvature tensor $K$ of
$\f$-K\"ahler-type. Then $K$ has the form
$$K=\nu L_1+\nu^* L_2,$$ where $\nu=\nu(K)$ and $\nu^*=\nu^*(K)=\nu(K^*)$
are the sectional curvatures of the totally real 2-planes
orthogonal to $\xi$ in $T_pM$, $p\in M$, with respect to $K$.
Moreover, $\M$ is of point-wise contact sectional curvatures of
the totally real 2-planes orthogonal to $\xi$ with respect to $K$.
\end{cor}

\section{Curvature tensor of the $\f$-canonical connection}

According to \cite{ManIv38}, a natural connection $D$ is called a
\emph{$\f$-cano\-nic\-al connection} on the manifold
$(M,\f,\xi,\allowbreak\eta,g)$ if the torsion tensor $T$ of $D$
satisfies the following identity:
\begin{equation*}\label{T-can}
\begin{split}
    &\bigl\{T(x,y,z)-T(x,\f y,\f z)
    -\eta(x)\left\{T(\xi,y,z)
    -T(\xi, \f y,\f z)\right\}\\[4pt]
    &-\eta(y)\left\{T(x,\xi,z)-T(x,z,\xi)-\eta(x)T(z,\xi,\xi)\right\}\bigr\}_{[y\leftrightarrow
    z]}=0.
\end{split}
\end{equation*}

Let us remark that the restriction the $\f$-canonical connection
$D$ of the manifold $\M$ on the contact distribution $\ker(\eta)$
is the unique canonical connection of the corresponding almost
complex manifold with Norden metric, studied in \cite{GaMi87}.

In \cite{ManGri2}, it is introduced a natural connection on $\M$,
defined by
\begin{equation}\label{fB}
    D_xy=\n_xy+\frac{1}{2}\bigl\{\left(\n_x\f\right)\f
y+\left(\n_x\eta\right)y\cdot\xi\bigr\}-\eta(y)\n_x\xi.
\end{equation}
In \cite{ManIv37}, the connection determined by \eqref{fB} is
called a \emph{$\f$B-connection}. It is studied for some classes
of $\M$ in \cite{ManGri2}, \cite{Man3}, \cite{Man4} and
\cite{ManIv37}. The $\f$B-connection is the odd-dimensional
analogue of the B-connection on the corresponding almost complex
manifold with Norden metric, studied for the class $\W_1$ in
\cite{GaGrMi}.

In \cite{ManIv38}, it is proved that the $\f$-canonical connection
and the $\f$B-connection coincide on the almost contact B-metric
manifolds in a class containing
$\F_1\oplus\F_4\oplus\F_5\oplus\F_{11}$.

According to \cite{ManGri2}, the necessary and sufficient
conditions $K$ to be a $\f$-K\"ahler-type tensor in $\F_i$
$(i=1,4,5,11)$ is the associated 1-forms $\ta$, $\ta^*$ and
$\om\circ\f$ to be closed. These subclasses we denote by $\F_i^0$
$(i=1,4,5,11)$.

Bearing in mind the second Bianchi identity
\[
\sx
\left\{\left(D_xK\right)(y,z)+K\left(T(x,y),z\right)\right\}=0,
\]
we compute the scalar curvatures for $K$ determined by
\[
\tau(K)=g^{ij}\rho(K)_{ij},\qquad
\tau^*(K)=\tau(K^*)={g}^{ij}\f_j^k\rho(K)_{ik},
\]
where $\rho(K)_{ij}$ is the Ricci tensor of $K$, and then we get
the following
\begin{lem}\label{lem:dtau}
For $\M$ in $\F_i^0$ $(i=1,4,5,11)$, the relations for the scalar
curvatures $\tau=\tau(K)$ and $\tau^*=\tau^*(K)$ of $K$ are:
\begin{equation}\label{dtau*}
\D\tau\circ\f=-\D\tau^*-\frac{1}{n}\left(\tau\ta+\tau^*\ta^*\right),
\quad
\D\tau^*\circ\f=\D\tau-\frac{1}{n}\left(\tau^*\ta-\tau\ta^*\right).
\end{equation}
\end{lem}

Obviously, bearing in mind \eqref{dtau*}, we obtain that the pair
$(\tau,\tau^*)$ on $\M$ is a $\f$-holomorphic pair of functions,
i.e. $\D\tau=\D\tau^*\circ\f$ and $\D\tau^*=-\D\tau\circ\f$, if
and only if the associated 1-forms $\ta$ and $\ta^*$ are zero.
Such one is the case for the class $\F_{11}$.

The system \eqref{dtau*} can be solved with respect to $\ta$ and
$\ta^*$ and then
\begin{equation}\label{h1h2}
\ta=-n\left\{\D f_1
            +\D f_2 \circ\f\right\}, \qquad
\ta^*=n\left\{\D f_1\circ\f-\D f_2\right\},
\end{equation}
where $f_1=\arctan\left(\tau^*/\tau\right)$,
$f_2=\ln\sqrt{\tau^2+\tau^{*2}}$.

Let us consider the complex-valued function $h=\tau+i\tau^*$ or in
polar form $h=|h|e^{i\al}$. Then we have
$|h|=\sqrt{\tau^2+\tau^{*2}}$,
$\al=\arctan\left(\tau^*/\tau\right)$.

Bearing in mind that $\Log h=\ln|h|+i\al$, then \eqref{h1h2} take
the following form:
\begin{equation}\label{ta-ln-al}
\ta=-n\left\{\D\al
            +\D(\ln|h|) \circ\f\right\}, \quad
\ta^*=n\left\{\D\al\circ\f-\D(\ln|h|)\right\}.
\end{equation}

So, we obtain the following

\begin{thm}\label{thm:tititau}
For $\M$ in $\F_i^0$ $(i=1,4,5)$, the associated 1-forms $\ta$ and
$\ta^*$ are derived by the non-$\f$-holomorphic pair of the scalar
curvatures $(\tau,\tau^*)$ of the $\f$-K\"ahler-type tensor $K$
for the $\f$-canonical connection $D$ by \eqref{ta-ln-al}.
\end{thm}

\begin{cor}\label{cor:tititau}\ \\
For $i=1$
\[
\ta=n\left\{\D\al\circ\f^2
            -\D(\ln|h|) \circ\f\right\}, \qquad
\ta^*=n\left\{\D\al\circ\f+\D(\ln|h|)\circ\f^2\right\};
\]
For $i=4$
\[
\ta=-n\D\al(\xi)\eta
            , \qquad
\ta^*=0;
\]
For $i=5$
\[
\ta=0, \qquad \ta^*=-n\D(\ln|h|)(\xi)\eta.
\]
\end{cor}

\section{Examples of almost contact manifolds with B-metric}


Let us consider $\R^{2n+2}=\left\{\left(u^1,\dots,
u^{n+1};v^1,\dots, v^{n+1}\right)\ \vert\ u^i, v^i \in\R\right\}$
as a complex Riemannian manifold with the canonical complex
structure $J$ and a metric $g$, defined by
$g(x,x)=-\delta_{ij}\lm^i\lm^j+\delta_{ij}\mu^i\mu^j$, where
$x=\lm^i\frac{\partial}{\partial
u^i}+\mu^i\frac{\partial}{\partial v^i}$. Identifying the point
$p=\left(u^1,\dots, u^{n+1};v^1,\dots, v^{n+1}\right)$ in
$\R^{2n+2}$ with its positional vector $Z$, in \cite{GaMiGr} it is
given a hypersurface $S$ defined by
\[
g(Z,JZ)=0,\quad g(Z,Z)=\cosh 2t,\quad t >0.
\]
The almost contact structure is determined by the conditions:
\[
\xi=\frac{1}{\cosh t}Z,\qquad Jx=\f x+\eta(x)J\xi,
\]
where $x,\f x \in
T_pS$ and $J\xi\in (T_pS)^\perp$. Then $(S,\f,\xi,\eta,g)$ is an
almost contact B-metric manifold in the class $\F_5$.

Consequently, we characterize $(S,\f,\xi,\eta,g)$ by means of
\cite{Man6}. We compute the following quantities for the
constructed $\F_5$-manifold:
\begin{equation}\label{exmpl}
\ta=0,\qquad \eta=\sinh t \D t,\qquad
\frac{\xi\ta^*(\xi)}{2n}=-\frac{{\ta^*}^2(\xi)}{4n^2}=\frac{1}{\cosh^2
t}.
\end{equation}

In \cite{ManGri1}, it is given that the 1-form $\ta^*$ on a
$\F_5$-manifold is closed if and only if
$x\ta^*(\xi)=\xi\ta^*(\xi)\eta(x)$. By virtue of \eqref{exmpl}, we
establish that $(S,\f,\xi,\eta,g)$ belongs to the subclass
$\F_5^0$, since $\D\ta^*=0$.

The condition for the second fundamental form of the hypersurface
$S$, given in \cite{GaMiGr}, the Gauss equation (\cite{KoNo}) and
the flatness of $\R^{2n+2}$ imply the following form of the
curvature tensor of $\n$
\[
R=-\frac{1}{\cosh^2 t}\pi_2.
\]
Then, taking into account \eqref{exmpl} and the form of the
curvature tensor $K$ of the $\f$-canonical connection in $\F_5^0$
(\cite{ManGri2})
\[
K=R+\frac{\xi\ta^*(\xi)}{2n}\pi_4+\frac{{\ta^*}^2(\xi)}{4n^2}\pi_1,
\]
we obtain
\begin{equation}\label{KL1}
K=\frac{1}{\cosh t}L_1.
\end{equation}
Since $L_1$ is a $\f$-K\"ahler-type tensor, then $K$ is also a
$\f$-K\"ahler-type tensor. Therefore, we have
\[
\nu(K)=K(e_1,e_2,e_2,e_1)=\frac{1}{\cosh^2 t}, \qquad
\nu^*(K)=K^*(e_1,e_2,e_2,e_1)=0,
\]
which confirm \thmref{thm:L5} and \corref{cor:K5}.

According to \eqref{KL1}, the scalar curvatures are
\[
\tau(K)=\frac{4n(n-1)}{\cosh^2 t},\qquad \tau^*(K)=0.
\]
Then, taking into account \eqref{exmpl}, the results for
$(S,\f,\xi,\eta,g)$ confirm also \lemref{lem:dtau},
\thmref{thm:tititau} and \corref{cor:tititau}.

\bigskip

\small{ \noindent
\textsl{M. Manev, M. Ivanova\\
Department of Algebra and Geometry\\
Faculty of Mathematics and Informatics\\
University of Plovdiv\\
236 Bulgaria Blvd\\
4003 Plovdiv, Bulgaria}
\\
\texttt{mmanev@uni-plovdiv.bg, mirkaiv@uni-plovdiv.bg} }


\begin{thebibliography}{99}

\bibitem{GaGrMi}
G. Ganchev, K. Gribachev, V. Mihova. \emph{B-connections and their
conformal invariants on conformally Kaehler ma\-nifolds with
B-metric}, Publ. Inst. Math. (Beo\-grad) (N.S.), vol. 42(56)
(1987), 107--121.


\bibitem{GaMi87}
G.~Ganchev, V.~Mihova. \emph{Canonical connection and the
canonical conformal group on an almost complex manifold with
B-metric}, Ann. Univ. Sofia Fac. Math. Inform., vol.~81 (1987),
195--206.


\bibitem{GaMiGr}
G. Ganchev, V. Mihova, K. Gribachev. \emph{Almost contact
manifolds with B-metr\-ic}, Math. Balkanica (N.S.), vol.~7,
no.~3-4 (1993), 261--276.


\bibitem{GrMeDj}
K. I. Gribachev, D. G. Mekerov, G. D. Djelepov. \emph{Generalized
B-manifolds}, Compt. rend. Acad. bulg. Sci., vol. 38, no. 3
(1985), 299--302.


\bibitem{KoNo}
S. Kobayshi, K. Nomizu. \emph{Foundations of differential
geometry}, Intersc. Publ., New York, 1969.



\bibitem{Man3}
M. Manev. \emph{Properties of curvature tensors on almost contact
manifolds with B-met\-ric}, Proc. of Jubilee Sci. Session of
Vassil Levsky Higher Mil. School, Veliko Tar\-novo, vol.~27
(1993), 221--227.

\bibitem{Man4}
M. Manev. \emph{Contactly conformal transformations of general
type of almost contact manifolds with B-metric. Applications},
Math. Balkanica (N.S.), vol. 11, no. 3-4 (1997), 347--357.

\bibitem{Man6}
M. Manev. \emph{Examples of almost contact manifolds with
B-met\-ric of some special classes}. Math. Educ. Math., Proc. 26th
Spring Conf. of UBM, Plovdiv, 1997, 153--160 (in Bulgarian).

\bibitem{Man-diss}
\textsc{M. Manev}. \emph{On the conformal geometry of almost
contact manifolds with B-met\-ric}, Ph.D. Thesis, Plovdiv, 1998.
(in Bulgarian)

\bibitem{Man31}
M. Manev. \emph{A connection with totally skew-symmetric torsion
on almost con\-tact ma\-ni\-folds with B-metric}, Int. J. Geom.
Methods Mod. Phys., vol. 9, no. 5 (2012),
arXiv:\allowbreak{}1001.3800.

\bibitem{ManGri1}
M. Manev, K. Gribachev. \emph{Contactly conformal transformations
of almost contact manifolds with B-metric}. Serdica Math. J., vol.
19 (1993), 287--299.

\bibitem{ManGri2}
M. Manev, K. Gribachev. \emph{Conformally invariant tensors on
almost con\-tact mani\-folds with B-metric}, Serdica Math. J.,
vol. 20 (1994), 133--147.

\bibitem{ManIv36}
M. Manev, M. Ivanova. \emph{A classification of the torsion
tensors on almost contact ma\-ni\-folds with B-metric},
arXiv:1105.5715.


\bibitem{ManIv37}
M. Manev, M. Ivanova. \emph{A natural connection on some classes
of almost contact mani\-folds with B-metric}, arXiv:1110.3023.

\bibitem{ManIv38}
M. Manev, M. Ivanova. \emph{Canonical-type connection on almost
contact manifolds with B-metric}, arXiv:1203.0137.


\bibitem{NakGri-97}
G. Nakova, K. Gribachev. \emph{Submanifolds of some almost contact
manifolds with B-metric with codimension two, I}, Math. Balk.,
vol. 11 (1997), 255--267.

\bibitem{NakMan14}
G. Nakova, M. Manev. \emph{Curvature tensors on some
five-dimen\-sional al\-most contact B-metric manifolds}, Math.
Educ. Math., Proc. 32nd Spring Conf. of UBM, Sunny Beach, 2003,
192--197.

\end{thebibliography}
\end{document}